\documentclass[12pt,a4paper]{article}

\usepackage{amsthm}
\usepackage{graphics,epsfig,color}

\def\bord{\partial}

\let\bydef\emph

\newtheorem{theo}{Theorem}[section]

\newtheorem{prop}[theo]{Proposition}
\newtheorem{lem}[theo]{Lemma}

\theoremstyle{definition}
\newtheorem{defi}[theo]{Definition}
\newtheorem{example}[theo]{Example}
\newtheorem{rem}[theo]{Remark}

\def\rr{\mathbf{R}}

\def\nn{\mathbf{N}}

\def\calx{\mathcal{X}}
\def\RR{\mathrm{R}}
\def\AA{\mathrm{A}}
\def\BB{\mathrm{B}}

\def\oo{\mathrm{O}}

\title{A class of open surfaces with algorithmically solvable homeomorphism problem}
\author{Sylvain Maillot}

\begin{document}

\maketitle

\begin{abstract}
We introduce a new class of possibly noncompact $n$-dimensional manifolds without boundary associated to finite data which we call topological automata. This class is large enough to contain many interesting examples of open 2-dimensional and 3-dimensional manifolds of interest to low-dimensional topologists. Our main result is that the homeomorphism problem in this class is decidable for $n=2$.\footnote{Partially supported by the Agence Nationale de la Recherche through Grant GTO ANR-12-BS01-0014.}
\end{abstract}

\section{Introduction}

A fundamental question in manifold topology is the \emph{homeomorphism problem}, which asks, in a given dimension $n$, whether there exists an algorithm to decide if two compact triangulated $n$-dimensional manifolds are homeomorphic. It is easy to show that the answer is yes for $n=1$ and $n=2$, using the classification of these manifolds. For $n=3$, the answer, although apparently not known in full  generality (see the comments in~\cite{afw:decision}), is yes for orientable, irreducible 3-manifolds. This relies on Perelman's geometrization theorem (\cite{Per1,Per3,Per2}, see also~\cite{kl:notes,Mor-Tia2,Cao-Zhu,b3mp:book}.) When $n\ge 4$, by a theorem of A.~A.~Markov~\cite{markov:homeomorphy} the answer is negative (see also~\cite{bhp:unsolvable}.) 

When the manifolds under consideration are no longer assumed to be compact, a difficulty arises: in each dimension $n\ge 2$ there are uncountably many open $n$-manifolds up to homeomorphism (see Appendix~\ref{sec:appendix}.)

In this article, we introduce for each $n$ a new class of possibly noncompact $n$-manifolds, called \emph{automata $n$-manifolds}, for which the homeomorphism problem makes sense. Our main result, Theorem~\ref{thm:main} below, states that this problem is algorithmically solvable for $n=2$.

\newpage

We first introduce the notion of \bydef{topological $n$-automaton}, which is a finite data set from which an automaton $n$-manifold will be constructed. The definition is as follows.

\begin{defi}
Let $n\ge 1$ be an integer. A \bydef{topological $n$-automaton} is a triple $\calx = (
(X_0,\ldots,X_p), (C_1,\ldots,C_p),(f_1,\ldots,f_q))$ where
\begin{itemize}
\item $p,q$ are nonnegative integers;
\item for each $k$, $X_k$ is a compact connected triangulated $n$-dimensional manifold-with-boundary, called a \bydef{building block}, or a \bydef{state};
\item for each $k\ge 1$, $C_k$ is a connected component of $\bord X_k$, hereafter called the \bydef{incoming boundary component} of $X_k$. The other boundary components of the building blocks, including all boundary components of $X_0$, if any, are called the \bydef{outcoming boundary components};
\item for each $i$, there exist $k,l$ with $k\le l$ such that
$f_i$ is a simplicial homeomorphism from some outcoming boundary component of  $X_k$ onto $C_l$; $f_i$ is called an \bydef{arrow};
\end{itemize}
subject to the condition that every outcoming boundary component is the domain of exactly one arrow.
\end{defi}

An example of topological $2$-automaton is given in Figure~\ref{fig:automaton}.

\begin{figure}[ht]
\begin{center}
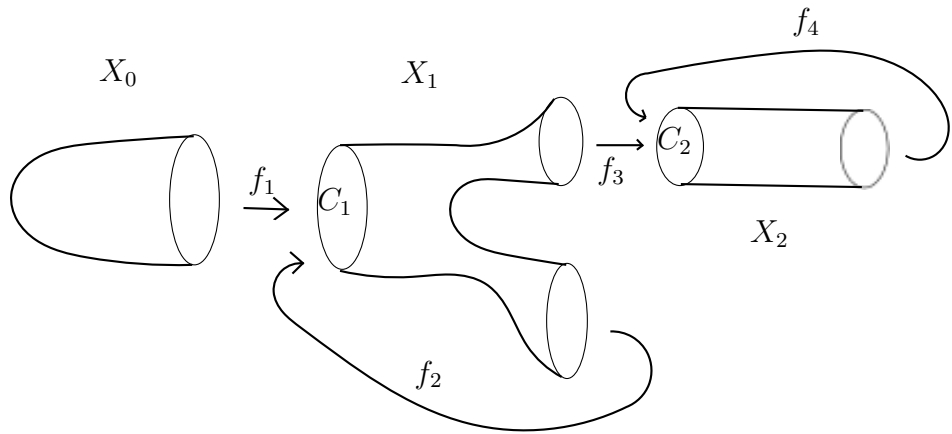
\end{center}
\caption {\label{fig:automaton} A topological $2$-automaton $\calx$}
\end{figure}

\begin{rem}
In order for decision problems about topological automata to be well-defined, one needs to be a bit more precise: letting $\calx$ be a topological $n$-automaton, fix a triangulation of each building block $X_k$, i.e.~a homeomorphism between $X_k$ and the geometric realization of a simplicial complex $A_k$, in such a way that each $f_i$ is conjugate to a simplicial map between some subcomplexes of the $A_k$'s. Then $\calx$ can be encoded by finite data that can be manipulated using one's favorite mathematical software.
\end{rem}

To any topological $n$-automaton $\calx$ we associate an $n$-dimensional manifold without boundary $M(\calx)$ by starting with $X_0$ and `following the arrows' to attach copies of the various building blocks. The idea is simple, but the formal definition is rather awkward, since we must specify an order in which the building blocks are attached, and keep track of the copies of the building blocks that have already been used. To this effect, note that by the last condition, the ordering of the $f_i$'s induces a linear ordering $L$ on the set of outcoming boundary components.

Let $\calx$ be a topological $n$-automaton. We first define inductively a sequence of triples $\{(M_s,\phi_s,L_s)\}_{s\in\nn}$ where for each $s$, $M_s$ is a compact $n$-manifold-with-boundary, $\phi_s$ is a collection of homeomorphisms from the components of $\bord M_s$ to outcoming boundary components of $\calx$, and $L_s$ is a linear ordering of the components of $\bord M_s$. For each $1\le k\le p$ we generate a sequence of copies of $X_k$ by setting $X_k^m := X_k\times \{m\}$ for each nonnegative integer $m$. By abuse of notation, we still denote by $f_i$ the homeomorphisms between the various components of the $\bord X_k^m$'s induced by $f_i$. We define $M_0$ to be $X_0$, $\phi_0$ to be the set of identity maps of the components of $\bord X_0$, and $L_0$ is induced by $L$.

Assume that the triple $(M_s,\phi_s,L_s)$ has been defined. Take the first component $C$ of $\bord M_s$, as given by $L_s$. Using $\phi_s$, we can identify $C$ to some outcoming boundary component of some $X_k$. Then by hypothesis there is a unique $f_i$ with domain $C$. Let $C_{k'}$ be the range of $f_i$. Take the first copy of $X_{k'}$ which has not already been used in the construction, and attach it to $M_s$ along $f_i$. Then repeat the operation for all components of $\bord M_s$, the order being determined by $L_s$. The resulting manifold is $M_{s+1}$. The set of homeomorphisms $\phi_{s+1}$ is determined by the identification of each added manifold with some building block. The ordering $L_{s+1}$ is deduced from $L_s$ and $L$ in a lexicographical manner.

By construction, there is a natural inclusion map from $M_s$ to $M_{s+1}$. Thus the $M_s$'s form a direct system. We define $M(\calx)$ as the direct limit of this system. See Figure~\ref{fig:develop} for an example.

We say that an $n$-manifold $M$ is an \bydef{automaton $n$-manifold} if it is homeomorphic to $M(\calx)$ for some topological $n$-automaton $\calx$.

\begin{figure}[ht]
\begin{center}
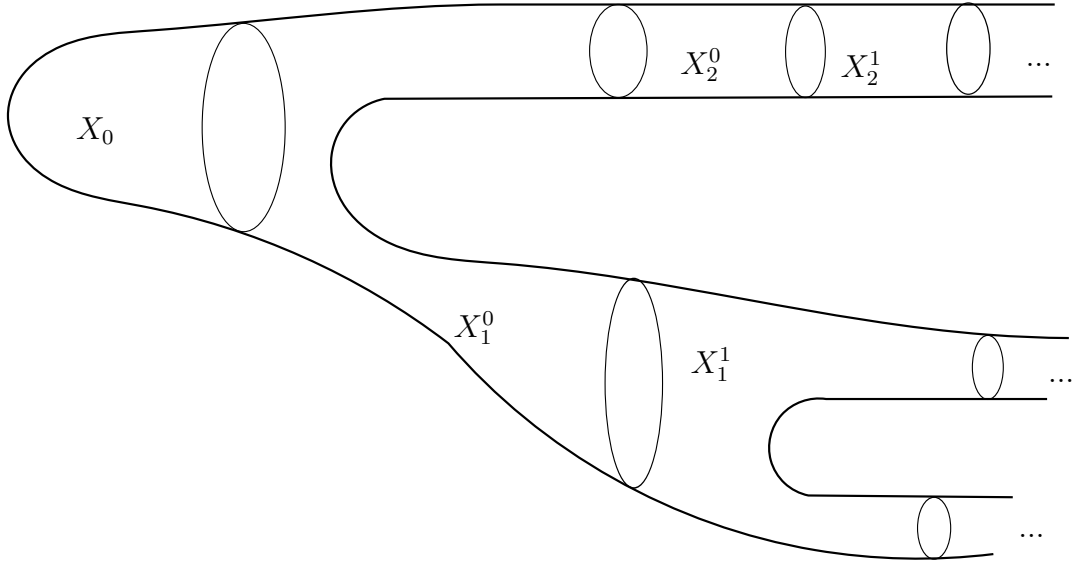
\end{center}
\caption {\label{fig:develop} The surface associated to $\calx$}
\end{figure}

\begin{rem}
The requirement that the building blocks be connected implies that $M(\calx)$ is connected, even if the graph underlying the automaton is disconnected. In fact, any vertex which is not accessible from $X_0$ can be discarded without changing the outcome of the construction. For example, consider a topological $2$-automaton $\calx$ with three states $X_0,X_1,X_2$, such that  $X_0$ is a disk, $X_1,X_2$ are annuli, and there are three arrows: one from the only boundary component of $X_0$ to the incoming boundary component of $X_1$, one from the outcoming boundary component of $X_1$ to the incoming boundary component of $X_1$, one from the outcoming boundary component of $X_2$ to the incoming boundary component of $X_2$. Then $M(\calx)$ is a plane obtained from stacking a copy of $X_0$ and a sequence of copies of $X_1$; the state $X_2$ plays no role at all.
\end{rem}

\begin{example}\label{ex:tame}
Let $M$ be an $n$-manifold that admits a manifold compactification, i.e.~there exists a compact manifold $\bar M$ such that $M$ is homeomorphic to $\bar M \setminus \bord {\bar M}$. Let $N_1,\ldots,N_p$ be the boundary components of $\bar M$. Then $M$ can be presented by a topological $n$-automaton with $p+1$ states as follows: $X_0=\bar M$, and for every $k\ge 1$, the state $X_k$ is $N_k\times [0,1]$ and the arrows are the obvious ones. In particular, 
\begin{enumerate}
\item[a)] Every closed $n$-manifold is an automaton manifold;
\item[b)] Every $1$-manifold is an automaton manifold;
\item[c)] Every finite-type surface is an automaton $2$-manifold.
\end{enumerate}
\end{example}

In this paper we are concerned with $n=2$, i.e.~surfaces. For us, the interesting case is that of surfaces of infinite type. These have been studied at least since the 1970s since they appear as leaves of foliations on compact 3-manifolds (see e.g.~\cite{ps:leaves,cc:endsets,ghys:generique}.) They are also studied from the point of view of Teichmüller theory (see e.g.~\cite{lp:infinite,matsuzaki:infinite} and the references therein) and translation surfaces \cite{randecker:wild}.

Recently they have been the subject of much attention through the study of their mapping class groups; we refer to the survey by Aramayona and Vlamis~\cite{av:big}.

Our main theorem is the following:

\begin{theo}\label{thm:main}
There is an algorithm which takes as input two topological $2$-automata $\calx_1,\calx_2$ and decides whether $M(\calx_1)$ and $M(\calx_2)$ are homeomorphic.
\end{theo}

The proof of Theorem~\ref{thm:main} relies essentially on the classification theorem for possibly noncompact surfaces, due to Ker\'ekj\'art\'o and Richards~\cite{richards:noncompact}. We now recall the statement of this theorem.

Let $M$ be a connected, possibly noncompact surface without boundary. The classical invariants which are used in the statement of the classification theorem are the genus $g(M)\in \nn \cup \{\infty\}$, the orientability class $o(M)$, and the triple $(E(M),E'(M),E''(M))$ where $E(M)$ is the space of ends of $M$, $E'(M)$ is the closed subspace of nonplanar ends, and $E''(M)$ is the closed subspace of nonorientable ends. Two such triples $(E,E',E'')$ and $(F,F',F'')$ are called \bydef{equivalent} if there is a homeomorphism $f:E\to E'$ such that $f(E')=F'$ and $f(E'')=F''$. Here is the statement of the classification theorem:

\begin{theo}[Classification of surfaces, Ker\'ekj\'art\'o, Richards~\cite{richards:noncompact}]
Let $M_1$ and $M_2$ be two connected surfaces without boundary. Then $M_1$ and $M_2$ are homeomorphic if and only if they have same genus and orientability class, and 
$(E(M_1),E'(M_1),E''(M_1))$ is equivalent to $(E(M_2),E'(M_2),E''(M_2))$.
\end{theo}

Recall that a surface $F$ is \bydef{planar} if every embedded circle in $F$ separates. A planar surface $F$ has genus $0$, is orientable, and $E'(F)=E''(F)=\emptyset$. Thus planar surfaces are classified by their space of ends.

A topological $2$-automaton is called \bydef{planar} if all of its building blocks are planar. This is equivalent to requiring that the associated surface should be planar.

We now give some examples and non-examples of automata surfaces. All of them are orientable, so $E''=\emptyset$. According to~\cite{av:big}, five infinite-type surfaces are sufficiently important to have received names in the literature. They are all automata surfaces:

\begin{example}
\begin{enumerate}
\item[a)] The \emph{Loch Ness monster surface} \cite{ps:leaves} has exactly one end, which is nonplanar. It is presented by an automaton with two states: $X_0$ is a disk; $X_1$ is an annulus with one handle.
\item[b)] \emph{Jacob's ladder surface}\footnote{We follow the terminology of Aramayona and Vlamis~\cite{av:big}, itself borrowed from Ghys~\cite{ghys:generique}; Philips and Sullivan~\cite{ps:leaves} used that name for a Riemannian surface that is homeomorphic to the Loch Ness monster surface, but to quasiisometric to it.} has exactly two ends, both nonplanar. It is presented by an automaton with three states: $X_0$ is an annulus; $X_1$ and $X_2$ are annuli with one handle.
\item[c)] The \emph{Cantor tree surface} (`arbre de Cantor'~\cite{ghys:generique}) is the complement of a Cantor set in $S^2$. It is presented by an automaton with two states: $X_0$ is a disk; $X_1$ is a pair of pants.
\item[d)] The \emph{blooming Cantor tree surface} (`arbre de Cantor fleuri'~\cite{ghys:generique}) has a Cantor set's worth of ends, none of which is planar. It is presented by an automaton with two states: $X_0$ is a disk; $X_1$ is a pair of pants with one handle.
\item[e)] The \emph{flute surface} \cite{basmajian:collar}: it is the planar surface with space of ends the ordinal $\omega+1$. This is the example depicted in Figures~\ref{fig:automaton} and~\ref{fig:develop}.
\end{enumerate}
\end{example}

In addition, the complement of a Cantor set in $\rr^2$ is presented by an automaton with three states: $X_0$ is an annulus; $X_1$ is an annulus; $X_2$ is a pair of pants. This is the example mentioned in a blog post by D.~Calegari and studied in J.~Bavard's PhD thesis~\cite{bavard:big} that sparked interest in `big mapping class groups'. Likewise, the three surfaces depicted on Figure 1 of the paper~\cite{mr:big} by Mann and Rafi are automata surfaces; informally, it is the very idea of a `surface that can be represented by a simple picture' which the notion of automaton surface endeavors to capture.

The simplest example of a surface which is \emph{not} an automaton surface is the planar surface with space of ends $\omega^\omega+1$ (see Example~\ref{ex:infinite} below.) More generally, if $M$ is a planar surface with countably many ends, then $M$ is an automaton surface if and only if each point of $E(M)$ has finite Cantor-Bendixson rank. For similar reasons, Example 2.3 in~\cite{mann:automatic} is not an automaton surface, and neither are most of the examples given in our  Appendix~\ref{sec:appendix}.

The paper is structured as follows: in Sections~\ref{sec:reduction} and~\ref{sec:planar}, we deal with the planar case, i.e.~we give an algorithm to decide whether the surfaces associated to two planar topological $2$-automata are homeomorphic. In Section~\ref{sec:reduction}, we show how to associate to a planar topological $2$-automaton $\calx$ a combinatorial object $T(\calx)$, called an \bydef{admissible tree}, such that the structure of the space of ends of $M(\calx)$ can be read off $T(\calx)$. In Section~\ref{sec:planar}, we introduce the notion of a \bydef{reduced tree}, which is a kind of normal form for $T(\calx)$, and prove Theorem~\ref{thm:main} in the planar case. In Section~\ref{sec:general}, we explain how to generalize this construction in order to prove Theorem~\ref{thm:main} in full generality. Finally, in Appendix~\ref{sec:appendix}, we indicate a construction of uncountably many open surfaces which are pairwise nonhomeomorphic.

\paragraph{Acknowledgements}
The author would like to thank Gilbert Levitt, Eric Swenson, Panos Papazoglu and J\'er\'emie Brieussel for fruitful conversations, and an anynomous referee for useful comments and in particular suggesting adding references.

\section{From automata to admissible trees}\label{sec:reduction}

In this section, we simply call \emph{2-automaton} a planar topological 2-automaton. To any such 2-automaton $\calx$, we shall associate a decorated graph $G(\calx)$ which contains enough information to recover the space of ends of the surface associated to $\calx$. Then we show how to associate to a decorated graph $G$ a decorated tree of a special form, called an \emph{admissible tree}, which also contains this information. This will allow us to reduce the topological classification of planar 2-automata to the classification of admissible trees up to some equivalence relation of topological nature. This latter classification problem will be solved in Section~\ref{sec:planar}.

The structure of this section is as follows: in Subsection~\ref{subsec:general}, we introduce decorated graphs, admissible trees and their associated topological spaces. The reader interested in the motivation can jump directly to Subsection~\ref{subsec:two} and refer back as needed. In Subsection~\ref{subsec:two}, we show how to pass from a 2-automaton to a decorated graph. Finally in Subsection~\ref{subsec:three} we show how to turn a decorated graph into an admissible tree.

\subsection{Decorated graphs and admissible trees}\label{subsec:general}

We introduce three formal symbols $*,\oo,\Theta$. We call \bydef{decorated graph} a pair $(G,f)$, where $G$ is a finite oriented graph endowed with a base vertex, called the \bydef{root}, and $f$ is a map from the set of all vertices of $G$ to the set $\{*,\oo,\Theta\}$, in such a way that the root has image $*$. We also assume that there is no circuit in $G$ as an oriented graph.

For simplicity, we will sometimes simply denote this graph by $G$, the decoration $f$ being understood. We say that a vertex is \bydef{of type} $*$, $\oo$, or $\Theta$ according to its image by $f$. A vertex different from the root is called an \bydef{ordinary vertex}.

We associate to any decorated graph $G$ a topological space $L(G)$ as follows: let $G$ be a decorated graph. Set $r := n_\oo + 2 n_\Theta$, where $n_\oo$ (resp.~$n_\Theta$) is the number of vertices of $G$ of type $\oo$ (resp.~of type $\Theta$). Consider the alphabet $B=\{b_1,\ldots, b_r\}$ and assign to each vertex $v$ of $G$ either one letter (if $v$ is of type $\oo$), or an unordered pair of letters (if $v$ is of type $\Theta$), or nothing. Let $L(G)$ be the set of infinite words $w$ on $B$ that can be obtained according to the following recipe: let $c=v_1\cdots v_l$ be a finite injective path of length at leat $2$ in $G$ starting from the root and respecting the orientations of the edges. For each $k\in \{2,\ldots,l-1\}$, choose a finite (possibly empty) word in the letter(s) assigned to $v_k$; then choose an infinite word in the letter(s) assigned to $v_l$; then $w$ is obtained by concatenating all those words. Finally, $L(G)$ is topologised as a subset of $B^\nn$ endowed with the product topology.

Of course, in this generality it is possible that some of the letters are never used (if some vertex is not accessible from the root.) It is even possible that $L(G)$ is empty. However, the space $B^\nn$ is either empty, a single point, or a Cantor set. Moreover, it is readily checked from the definition that $L(G)$ a closed subset of $B^\nn$. Thus $L(G)$ is compact, metrizable and totally disconnected. We shall see later that when $G$ is obtained from a 2-automaton $\calx$ by a certain construction, then $L(G)$ is homeomorphic to the space of ends of $M(\calx)$.

Next we define a special class of decorated graphs which will play an important role.

\begin{defi}
An \bydef{admissible tree} is a decorated graph which has the following properties:
\begin{enumerate}
\item It is a tree.
\item The only vertex of type $*$ is the root.
\item All edges point away from the root. 
\end{enumerate}
\end{defi}

It follows that the orientations of edges is uniquely determined by the underlying structure, so in the pictures, we will not represent these orientations. The convention is that the edges are oriented from left to right.

We will use some standard terminology concerning rooted trees. The \bydef{leaves} are the degree~1 vertices which are ordinary (i.e.~different from the root.) We use the words `ancestor', `descendant', `son', `father', and `sibling' as in genealogy.

We now give some information on the space $L(T)$ associated to an admissible tree $T$. First, the notion of convergence associated to our topology is that of compact convergence; more precisely, if $w$ is a point of $L(T)$ and $(w_n)$ a sequence of points of $L(T)$, then $(w_n)$ converges to $w$ iff for each $m\in \nn$, we have that the prefix of length $m$ of $w_n$ is eventually equal to the prefix of length $m$ of $w$.

Let $T$ be an admissible tree, and $v$ be an ordinary vertex of $T$. We denote by $L_v$ the subset of $L(T)$ consisting of the words that have infinitely many letters associated to $v$. By definition, $L(T)$ is the disjoint union of the $L_v$'s. If $v'$ is a descendant of $v$, then any $w\in L_v$ is a limit of a sequence $(w_n)$ of points of $L_{v'}$, obtained by taking the $n$ first letters of $w$ and completing by an infinite word in $L_{v'}$.

We denote by $T(v)$ the subtree consisting of $v$ and its descendants. We call this tree the \bydef{subtree generated by} $v$. Note that $T(v)$ is not an admissible tree (or even a decorated graph) since its ``root'' is $v$, which does not have type $*$. However, such subtrees will play an important role in inductive arguments. It follows from the definition of the topology that the union $U$ of the $L_{v'}$'s for $v'\in T(v)$ is an open subset of $L(T)$. In particular, it is a neighborhood of any point of $L_v$. This gives a partial converse to the above observation: if $w\in L_v$ and $(w_n)$ is a sequence converging to $w$, then $w_n$ eventually belongs to $U$.

\subsection{From planar 2-automata to decorated graphs}\label{subsec:two}

Let $\calx$ be a planar 2-automaton. The goal of this subsection is to construct a decorated graph $G(\calx)$ such that $L(G(\calx))$ is homeomorphic to $E(M(\calx))$.

We start by forming the directed graph $G'$ whose vertices are the building blocks and whose edges are the arrows.
An arrow is called a \bydef{loop} if its domain and range lie in the same building block. 

In order to describe the space of ends of the surface $M(\calx)$,
take a finite alphabet $A$ with as many letters as arrows in $\calx$, and assign a letter to each arrow. We give the set $A^\nn$ of infinite words in $A$ the product topology. The space of ends of the surface $M(\calx)$ is homeomorphic to the subset $E\subset A^\nn$ consisting of words which can be read by starting from $X_0$ and following the arrows.

To each prefix $w=a_1\cdots a_n$ of each word in $E$ we can associate the number $k$ such that $w$ leads to $X_k$, and the subset $E_w$ of $E$ of words whose prefix is $w$ and whose other letters correspond to loops around $X_k$. There are three cases: if there is no such loop, then $E_w$ is empty; if there is one, then $E_w$ is a singleton; otherwise, $E_w$ is a Cantor set, no matter how many loops around $X_k$ there are. This follows from the well-known fact that the Cantor set is the only totally disconnected compactum without isolated points. Hence we see that the number of loops around some building block is irrelevant as soon as it is greater than or equal to $2$.

We are now ready to define the decorating graph $G(\calx)$:
there is a vertex of $G(\calx)$ for each building block $X_k$, the root being $X_0$. Edges of $G(\calx)$ correspond to arrows $f_i$ which are not loops. The image of a vertex by $f$ is $*$ (resp.~$\oo$, resp.~$\Theta$) if there is no loop (resp.~one loop, resp.~two or more loops) around the corresponding building block. The condition $k\le l$ in the fourth point of the definition of a topological automaton ensures that the oriented graph $G(\calx)$ does not contain any circuit. (There may, however, exist circuits in $G(\calx)$ as an unoriented graph.)

It follows from the above description of $E(M(\calx))$ that this space is homeomorphic to the space $L(G(\calx))$ defined in Subsection~\ref{subsec:general}.

\subsection{From decorated graphs to admissible trees}\label{subsec:three}

Let $G$ be a decorated graph. We describe a procedure which yields an admissible tree $T$ such that $L(G)$ is homeomorphic to $L(T)$. We do not need to know that $G$ comes from an automaton, although this is the case we are interested in.

The construction of $T$ is done in the following three steps. At each step, one checks that the homeomorphism type of $L(T)$ does not change.

\paragraph{Step 1} If some vertices of $G$ are not accessible from the root, then we remove them. 

\paragraph{Step 2} If $G$ is not a tree, then it contains two distinct injective paths connecting the root to some vertex $y$. We choose a `minimal' such $y$, i.e.~farthest from the root as possible, so that the subgraph $Y$ consisting of $y$ and the vertices accesible from it is a tree.

Let $e_1,\ldots,e_p$ be the edges going `into' $y$. For each $i$, denote by $x_i$ the other end of $e_i$. We construct a new decorated graph by introducing for each $2\le i\le p$ a copy $(Y_i,y_i)$ of $(Y,y)$, deleting $e_i$ and replacing it by an edge $e'_i$ connecting $x_i$ to $y_i$. An example of this operation is represented in Figure~\ref{fig:treeing}. 

By repeating this operation finitely many times, we obtain a decorated graph $T'$ which is a tree.

\begin{figure}[ht]
\begin{center}
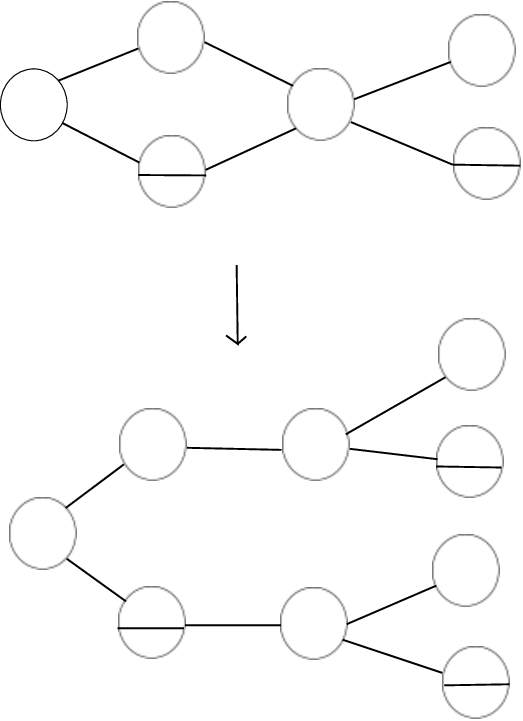
\end{center}
\caption {\label{fig:treeing} Making the graph $G$ into a tree}
\end{figure}

\paragraph{Step 3} If $T'$ is not admissible, pick a vertex $x$ of type $*$ which is not the root. Let $e$ be the edge leading to $x$ (by the previous construction, there must be exactly one). Let $x'$ be the initial vertex of $e$, and let $e_1,\ldots,e_p$ be the edges with initial vertex $x$ (if any.) Modify $T'$ by deleting the vertex $x$ and the edge $e$, and replacing each $e_i$ by an edge $e'_i$ with initial vertex $x'$ and same terminal vertex as $e_i$. See Figure~\ref{fig:admissible} for an example of this operation. (Informally, it corresponds to `collapsing' the edge $e$.) Repeating this finitely many times, we obtain the admissible tree $T$.

\begin{figure}[ht]
\begin{center}
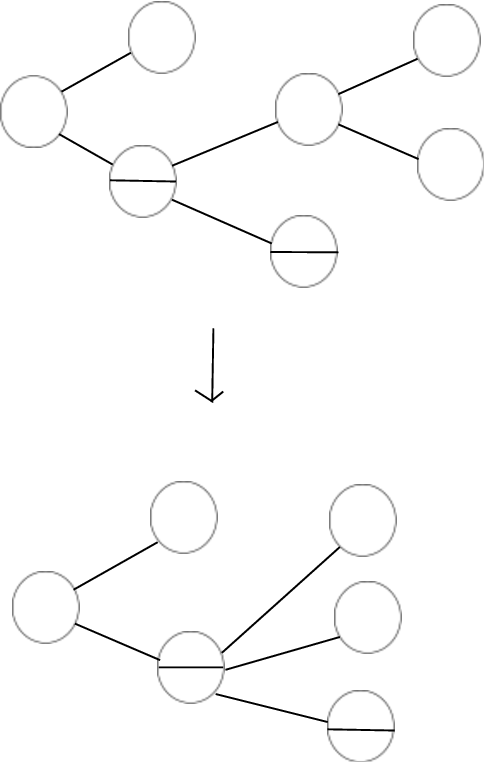
\end{center}
\caption {\label{fig:admissible} Making the tree admissible}
\end{figure}

To sum up, starting with a $2$-automaton $\calx$, we can produce a decorated graph $G=G(\calx)$ such that $L(G)$ is homeomorphic to $E(M(\calx))$, and then an admissible tree $T$ such that $L(T)$ is homeomorphic to $L(G)$, hence also to $E(M(\calx))$.

Given two planar automata $\calx_1,\calx_2$, we can compute the associated admissible trees $T_1,T_2$, and our problem is reduced to checking whether or not $L(T_1)$ and $L(T_2)$ are homeomorphic. A solution of the latter problem is given in the next section. For the general case, we need more structure, see Section~\ref{sec:general}.

\section{Classification of admissible trees}\label{sec:planar}

\subsection{Reduced admissible trees}\label{subsec:reduced}

Let $T_1$, $T_2$ be two admissible trees. We say that $T_1$ and $T_2$ are \bydef{topologically equivalent}, or just \bydef{equivalent}, if the associated spaces $L(T_1)$ and $L(T_2)$ are homeomorphic. We say that they are \bydef{isomorphic} if there is a type-preserving bijection between the set of vertices of $T_1$ and that of $T_2$ which respects the graph structure.

Determining whether two given admissible trees are (topologically) equivalent is the problem we want to solve. The problem of determining whether two given admissible trees are isomorphic is combinatorial in nature and easy to solve algorithmically simply by enumerating all maps from the set of vertices of $T_1$ to that of $T_2$ until a suitable bijection is found (or not.)

We will show how to modify an admissible tree without changing its topological equivalence class, until it belongs to a special class of admissible trees, called \emph{reduced}, for which topological equivalence will turn out to be equivalent to isomorphism. In order to motivate the construction, we first give some simple examples of pairs of admissible trees which are equivalent but fail to be isomorphic.

Let $T_1$ be a tree with two vertices: $*$ and its son $v$ of type $\Theta$. The space $L(T_1)$ is the space of all infinite words on a two-letter alphabet, which is a Cantor set.
Let $T_2$ be a tree with three vertices: $*$, $v_1$ and $v_2$, where both $v_1$ and $v_2$ are of type $\Theta$ and are sons of $*$. Then $L(T_2)$ is a disjoint sum of two Cantor sets, which is homeomorphic to a Cantor set. Thus $T_1,T_2$ are equivalent but nonisomorphic.

Let $T_3$ be a tree with three vertices: $*$, its only son $v$, and the only son $w$ of $v$, with both $v,w$ of type $\oo$. The space $L(T_3)$ can be described as follows: the alphabet is $\{a,b\}$ with the letter $a$ corresponding to $v$ and $b$ to $w$. Then the elements of $L(T_3)$ are the (infinite) word $aaa\cdots$ and all words of the type $a^kbbb\cdots$ for $k\in\nn$. Topologically, this is a space with a single accumulation point and a sequence of isolated points converging to it.

Let $T_4$ be a tree with vertices $*,v,w_1,w_2$, with $v$ the only son of $*$, $w_1,w_2$ two sons of $v$, and every ordinary vertex has type $\oo$. The space $L(T_4)$ consists of an accumulation point and two sequences of isolated points converging to it. This space is homeomorphic to $L(T_3)$ although $T_3$ is not isomorphic to $T_4$.

Finally if $T_5$ is a tree with vertices $*$, its only son $v$, and the only son $w$ of $v$, with both $v,w$ of type $\Theta$, then $L(T_5)$ has no isolated point. Hence it is a Cantor set. More generally, any tree all of whose leaves are of type $\Theta$ have the same property, which produces many nonisomorphic equivalent trees.

We now introduce three moves that can be used to simplify an admissible tree without changing its equivalence class.
Let $T$ be an admissible tree. Recall that when $v$ is a vertex of $T$, we denote by $T(v)$ the subtree consisting of $v$ and its descendants.

\paragraph{Move 1} Let $v$ be a vertex of $T$ of type $\Theta$. Let $v'$ be the father of $v$. Assume that $v'$ is not the root, and that $v$ is the only son of $v'$. Remove the edge between $v$ and $v'$ and replace $v'$ by $v$.

\begin{figure}[ht]
\begin{center}
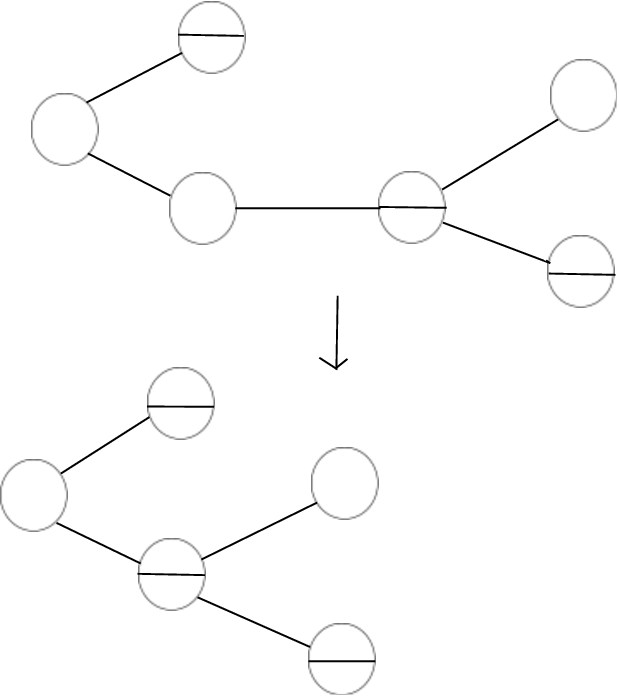
\end{center}
\caption {\label{fig:move1} Move 1}
\end{figure}

\paragraph{Move 2} Let $v_1$ be a vertex of $T$. Let $v_2,v_3$ be two descendants of $v_1$ such that $v_2$ is a son of $v_1$, and $v_3$ is not. Assume that the subtrees $T(v_2)$ and $T(v_3)$ are isomorphic. Remove $T(v_2)$.

\begin{figure}[ht]
\begin{center}
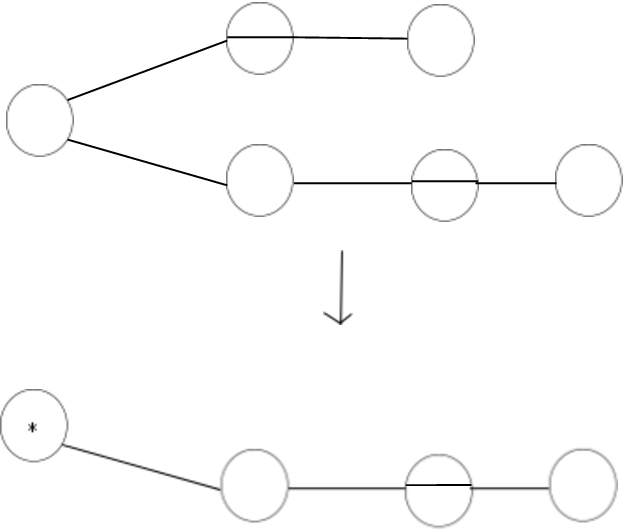
\end{center}
\caption {\label{fig:move2} Move 2}
\end{figure}

\paragraph{Move 3} Let $v_1$ be a vertex of $T$. Let $v_2,v_3$ be two distinct sons of $v_1$. Assume that $v_1$ is not the root, or that both $v_2$ and $v_3$ are of type $\Theta$. Further assume that the subtrees $T(v_2)$ and $T(v_3)$ are isomorphic. Remove $T(v_3)$.

\begin{figure}[ht]
\begin{center}
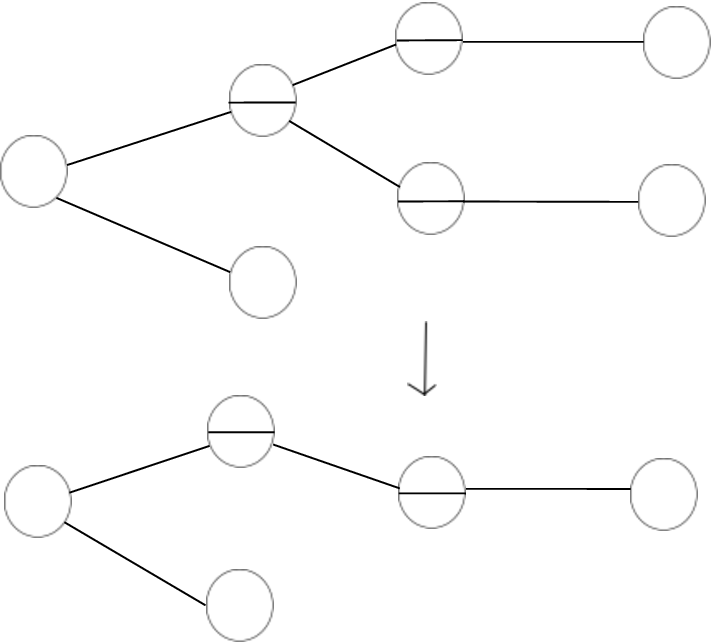
\end{center}
\caption {\label{fig:move3} Move 3}
\end{figure}

Simple examples of the three moves are represented in Figures~\ref{fig:move1}, \ref{fig:move2}, and~\ref{fig:move3} respectively.

\begin{defi}
Let $T$ be an admissible tree. For $i\in\{1,2,3\}$, we say that $T$ has Property~$(\RR_i)$ if Move $i$ cannot be performed. We say that $T$ is \bydef{reduced} if it has property $(\RR_i)$ for all $i$.
\end{defi}

We leave to the reader the tedious but elementary task of checking that Moves~1--3 do not change the homeomorphism type of the associated space $L(T)$. We then have the following proposition:
\begin{prop}\label{prop:reducing}
There is an algorithm which to any admissible tree $T$ associates a reduced admissible tree equivalent to $T$.
\end{prop}

\begin{proof}
Starting with $T$, apply Move 1 as much as possible, then Move 2 as much as possible, then Move 3 as much as possible, then Move 1 again etc. Each time a move is applied, the number of vertices of the tree goes down, hence the process eventually stops.
\end{proof}

The main result of this section is

\begin{theo}\label{thm:reduced}
Two reduced admissible trees $T_1,T_2$ are equivalent if and only if they are isomorphic.
\end{theo}

The `if' part is clear, so we only have to prove the `only if' part. To this end, we develop some theory in Subsection~\ref{sub:dec}, to be applied in Subsection~\ref{sub:theproof}.

\subsection{Decomposing a totally disconnected space}\label{sub:dec}
To provide motivation for the notions introduced in this subsection, we first define inductively the \bydef{combinatorial depth} of an ordinary vertex in an admissible tree.  Let $T$ be an admissible tree. Recall that the \bydef{leaves} of $T$ are  the vertices of $T$ which are ordinary and have degree~$1$ (i.e., have no descendant.) We say that an ordinary vertex of $T$ has \bydef{combinatorial depth}~$0$ if it is a leaf of $T$. Let $T'$ be the admissible tree obtained from $T$ be removing its leaves. Then the leaves of $T'$ are declared to have combinatorial depth $1$, and so on. Since $T$ is finite, every ordinary vertex of $T$ has combinatorial depth $p$ for some $p\in\nn$.

We now define the (topological) \bydef{depth} of a point in a compact, metrizable, totally disconnected space $X$. When $X=L(T)$ for a reduced admissible tree $T$, this will be \emph{a posteriori} closely related to combinatorial depth in $T$, although we will not prove that fact, since we will not need it. The important point is that the definition is topological rather than combinatorial.

First we need some preliminary definitions. For brevity, we simply say that $X$ is a \bydef{space} if it is a compact, metrizable, totally disconnected topological space. A \bydef{pointed space} is a pair $(X,x)$ where $X$ is a space and $x$ a point of $X$.

\begin{defi}
Two pointed spaces $(X,x)$ and $(X',x')$ are \bydef{topologically equivalent} (or just \bydef{equivalent}) if there exists a triple $(U,U',f)$, where $U$ is a neighborhood of $x$, $U'$ is a neighborhood of $x'$, and $f$ is a homeomorphism from $U$ to $U'$ such that $f(x)=x'$.

We sometimes say that $x$ is \bydef{equivalent} to $x'$ if there is no ambiguity concerning the ambient spaces.
\end{defi}

We define inductively on a nonnegative integer $p$ what it means for a point in a space to have \bydef{depth} $p$:

\begin{defi}
Let $(X,x)$ be a pointed space. We say that $x$ has \bydef{depth} $0$ if it is isolated, or has a neighborhood homeomorphic to the Cantor set.

Assume that the phrase ``to have depth $k$'' has been defined for all $k$ less than some number $p$. Let $Y$ be the set of points of $X$ that do not have depth $<p$. Then a point $x\in X$ is said to have \bydef{depth} $p$ if it belongs to $Y$, and either it is isolated in $Y$, or it has a neighborhood $K$ in $Y$ which is a Cantor set, and such that for every $y\in K$, the pairs $(X,x)$ and $(X,y)$ are equivalent.

We say that $x$ has \bydef{finite depth} if it has depth $p$ for some $p\in\nn$. Otherwise it has \bydef{infinite depth}.
\end{defi}

\begin{rem}
\begin{enumerate}
\item[a)] Let $T$ be an admissible tree. We shall see shortly that every point of $L(T)$ has finite depth, which can be bounded from above by the combinatorics of $T$. This is a rather straightforward consequence of the finiteness of $T$. Obtaining lower bounds for the depth is more difficult and this is what much of the proof of Theorem~\ref{thm:reduced} is about.
\item[b)] Let $X$ be a space and $p$ be a nonnegative integer. Then it is easy to see by induction on $p$ that the set of points of $X$ of depth at most $p$ (resp.~at least $p$) is open (resp.~closed).
We shall denote this subset by $X^{\le p}$ (resp.~$X^{\ge p}$.) 
\item[c)] The notion of depth is similar to the Cantor-Bendixson rank; in fact, for countable spaces it is equal to the Cantor-Bendixson rank whenever it is finite (see below.)
\end{enumerate}
\end{rem}

We have the following useful lemma, which tells us that depth depends only on the local topology:
\begin{lem}\label{lem:local}
For every $p\in\nn$ the following holds: let $(X,x)$ and $(X',x')$ be pointed spaces which are equivalent to each other and such that $x$ has depth $p$; then $x'$ also has depth $p$.
\end{lem}

\begin{proof}
The proof is by induction on $p$. For $p=0$ the statement is immediate from the definitions.

Suppose that the statement is true for all $p'<p$. Let $(X,x)$ and $(X',x')$ be equivalent pointed spaces such that $x$ has depth $p$. Let $(U,U',f)$ be a triple as in the definition of topological equivalence. Since $X^{\le p}$ is open, we may, up to replacing $U$ and $U'$ by smaller neighborhoods, assume that $U$ contains only points of depth $\le p$.

If $x'$ had depth $<p$, then applying the induction hypothesis and reversing the roles of $X,X'$, we could deduce that $x$ has depth $<p$. Hence $x'$ has depth at least $p$ (possibly infinite.) Denote by $V$ (resp.~$V'$) the subset of $U$ (resp.~$U'$) consisting of points of depth $<p$. From the induction hypothesis, we know that $f$ induces a bijection from $V$ to $V'$.

By hypothesis, the point $x$ belongs to $X^{\ge p}$, the point $x'$ belongs to $X'^{\ge p}$, and the pairs $(X^{\ge p},x)$ and $(X'^{\ge p},x')$ are equivalent. If $x$ is isolated in $X^{\ge p}$, then $x'$ is isolated in $X'^{\ge p}$, so $x'$ has depth $p$. Otherwise, there is a Cantor set $K\subset X^{\ge p}$ which is a neighborhood of $x$, and all of whose points are equivalent to $x$ in $X$. We may assume that $K\subset U$ and take its image $K'$ under $f$. Thus we have produced a Cantor set $K'$ containing $x'$, and with the property that every $y'\in K'$ is equivalent to $f^{-1}(y')$, hence to $x'$, hence to $x$. This shows that $x'$ has depth $p$.
\end{proof}

\begin{rem}
As a by-product of the proof we see that if $(X,x)$ and $(X',x')$ are equivalent pairs, then we can choose the triple $(U,U',f)$ so that $U$ (resp.~$U'$) contains only points of depth $\le p$, and $f$ is depth-preserving.
\end{rem}

Let us illustrate the notions of depth and topological equivalence on a few examples. First, when $X$ does not contain any Cantor set (equivalently, when $X$ is countable), the points of depth $0$ are exactly the isolated points, the points of depth 1 are the accumulation points which are isolated among accumulation points, and so on. By taking appropriate countable ordinals with the order topology, one can produce for each $p\in\nn$ examples where there exists a point of depth $p$, and all points have depth $\le p$. One can also produce examples which have points of infinite depth:

\begin{example}\label{ex:infinite}
The ordinal $X=\omega^\omega+1$ has points of arbitrarily large finite depths, and exactly one point of infinite depth. By Remark~\ref{rem:infinite} below, there is no admissible tree $T$ such that $L(T)$ is homeomorphic to $X$. This in turn implies that if $f:X\to S^2$ is an embedding, then $S^2\setminus f(X)$ is a surface which is not an automaton surface.
\end{example}

We now turn to an example which does contain Cantor sets.

\begin{example}
Let $T_6$ be an admissible tree with four vertices $*,v_1,v_2,v_3$, where $v_1$ is the only son of $*$ and $v_2,v_3$ are sons of $v_1$, and such that $v_1,v_2$ have type $\oo$ and $v_3$ has type $\Theta$. This is one of the simplest reduced trees. We introduce the alphabet $\{a,b,c,d\}$ with the letter $a$ corresponding to $v_1$, the letter $b$ to $v_2$, and the letters $c,d$ to $v_3$. Then $L(T_6)$ is partitioned into three subsets $\{a^\infty\}$, $U$ and $V$, where $a^\infty$ is the infinite word $aaa\cdots$, the set $U$ contains all words consisting of zero-or-more $a$'s followed by infinitely many $b$'s, and $V$ contains all words consisting of zero-or-more $a$'s followed by an infinite word on the letters $c,d$.

The points of $U$ are exactly the isolated points, which have depth 0. Each point of $V$ is contained in some clopen Cantor set, thus also has depth 0. The point $a^\infty$ has none of the above properties, so it does not have depth 0. Hence $L(T_6)^{\ge 1}=\{a^\infty\}$, and $a^\infty$, being isolated in $L(T_6)^{\ge 1}$, has depth 1 in $L(T_6)$.
\end{example}

\begin{rem}
In this example, the topological equivalence classes are exactly the sets $\{a^\infty\}$, $U$ and $V$. If $X$ is a countable space and $x\in X$ a point of depth $1$, then the pair $(L(T_6,a^\infty))$ is not equivalent to the pair $(X,x)$. Indeed, $a^\infty$ is the limit of a sequence of nonisolated points of depth 0, whereas $x$ is not.
\end{rem}

More generally, if $T$ is an admissible tree and $v$ a leaf of $T$, then the words of $L(T)$ which contain infinitely many letters corresponding to $v$ have depth 0. The converse does not hold in general. For instance the space $L(T_5)$ considered in Subsection~\ref{subsec:reduced}  is a Cantor set, so every point has depth 0. We shall see, however, that this phenomenon cannot happen if $T$ is reduced; this will be an important step in the proof of Theorem~\ref{thm:reduced}.

\subsection{Proof of Theorem~\ref{thm:reduced}}\label{sub:theproof}
Let $T$ be an admissible tree. Let $v$ be an ordinary vertex of $T$. Recall that we denote by $T(v)$ the subtree generated by $v$ (i.e.~consisting of $v$ and its descendants,) and by $L_v$ the set of words in $L(T)$ which have infinitely many letters corresponding to $v$. In the sequel, we shall abusively denote by $L(T(v))$ the language associated to the admissible tree obtained from $T(v)$ by adding a root. Let us denote by $S_v$ the subset of $L(T(v))$ consisting of elements which have at most finitely many letters corresponding to $v$. The following lemma shows that it makes sense to talk about the depth of an ordinary vertex, and that topological equivalence induces an equivalence relation between ordinary vertices of admissible trees.
\begin{lem}\label{lem:ind}
Let $T$ be an admissible tree. For every ordinary vertex $v$, any two elements of $L_v$ are topologically equivalent. In particular, they have have the same depth.
\end{lem}

\begin{proof}
Let $v$ be an ordinary vertex of $T$. Let $w,w'$ be words in $L_v$. By definition, $w$ is equal to $w_1 w_2$ where $w_1$ is a (possibly empty) finite word, and $w_2$ is an infinite word on the letter(s) associated to $v$. Similarly, $w'$ has a decomposition $w'_1 w'_2$.

\paragraph{Case 1} The vertex $v$ has type $\oo$.

Let $b_k$ be the letter associated to $v$. We have $w_2=w'_2=b_k^{\infty}$. If $v$ is a leaf, then both $w$ and $w'$ are isolated, and thus equivalent. Otherwise, let $U$ be the union of $\{w\}$ with the set of words of the form $w_1w_5$ with $w_5\in S_v$. Likewise, let $U'$ be the union of $\{w'\}$ with the set of words of the form $w'_1w_5$ with $w_5\in S_v$. It follows from the definition of the topology on $L(T)$ that $U$ (resp.~$U'$) is an open neighborhood of $w$ (resp.~$w'$). Moreover, the map $f:U\to U'$ defined by setting $f(w):=w'$, and for every $w_5\in S_w$, $f(w_1w_5):=w'_1w_5$ is a homeomorphism. Hence $w$ is equivalent to $w'$.

\paragraph{Case  2} The vertex $v$ has type $\Theta$.

Let $b_k,b_l$ be the letters associated to $v$. Let $T_{kl}$ be the rooted infinite dyadic tree with edges labeled by $b_k,b_l$. The set $F$ of finite words on $\{b_k,b_l\}$ (i.e.~the free monoid on $\{b_k,b_l\}$) can be identified with the set of finite injective paths in $T_{kl}$ starting from the root. Likewise, the set $K$ of infinite words on the same letters can be identified with the set of infinite injective paths in $T_{kl}$ starting from the root. There exists an automorphism $\phi$ of $T_{kl}$ (as a rooted tree) which takes every finite prefix of $w_2$ to the prefix of $w'_2$ of same length~(cf.~\cite[Section 3]{grigorchuk:just}.) This automorphism induces a self-homeomorphism (still denoted by $\phi$) of the Cantor set $K$ which takes $w_2$ to $w'_2$.

We now argue as in Case 1. Let $U$ (resp.~$U'$) be the set of words of the form $w_1w_5$ (resp.~$w'_1w_5$) with $w_5\in L(T(v))$. Then $U$ is a neighborhood of $w$, $U'$ is a neighborhood of $w'$, and we can define a homeomorphism $f:U\to U'$ in the following way: if $w_5\in K$, then $f(w_1w_5):=w'_1\phi(w_5)$ (in particular, $f(w)=w'$ as required); otherwise $w_5$ is of the form $w_3w_4$ with $w_3\in F$ and $w_4$ is some infinite word on the letters associated to descendants of $v$, and we set $f(w_1w_3w_4):=w'_1\phi(w_3)w_4$.
\end{proof}

\begin{rem}
A variation on this argument proves that if $T_1,T_2$ are admissible trees, $v_1$ is an ordinary vertex of $T_1$, and $v_2$ is an ordinary vertex of $T_2$ such that $T_1(v_1)$ is isomorphic to $T_2(v_2)$, then $v_1$ is equivalent to $v_2$. The converse is not true in general; however, we shall see that it is true if both $T_1$ and $T_2$ are reduced. This will be a key step in the proof.
\end{rem}

\begin{rem}\label{rem:infinite}
Let $T$ be an admissible tree. If $x\in L_v$ for some leaf $v$ of $T$, then $x$ is isolated or belongs to some Cantor clopen subset of $L$. Thus $v$ has depth $0$. If $v$ is an ordinary vertex of $T$ which is not a leaf, and all descendants of $v$ have finite depth, then $v$ has finite depth, equal to $M$ or $M+1$, where $M$ is the maximum of the depths of the descendants of $v$. This implies that every ordinary vertex has finite depth, and depth is nonincreasing along paths in the tree starting from the root and going to the leafs.
\end{rem}

\begin{lem}\label{lem:zero}
Let $T$ be a reduced tree. Then the vertices of depth $0$ of $T$ are exactly its leaves.
\end{lem}

\begin{proof}
We have already remarked that the leaves have depth 0. Let us prove the converse by contradiction. Let $T$ be a reduced tree, and $v$ be a vertex of depth 0 which is not a leaf. Assume that $v$ has minimal combinatorial depth among such vertices, i.e.~is as close to the leaves as possible. All the descendants of $v$ have depth 0, so by minimality, $v$ has combinatorial depth $1$.

If $v$ has exactly one son $v'$, then $v'$ must be of type $\Theta$; otherwise every point of $L_v$ would be the limit of some sequence of isolated points, hence of nonzero depth. This contradicts Property~$(\RR_1)$.

Thus $v$ has at least two sons. By Property~$(\RR_3)$, it has at most (hence exactly) two sons $v'$, $v''$, and these have different types. Now if, say, $v'$ has type $\oo$, then again every point of $L_v$ would be the limit of sequence of isolated points, giving a contradiction.
\end{proof}

\begin{prop}\label{prop:key}
For every integer $p\ge 0$, the following assertions hold:
\begin{enumerate}
\item[$(\AA_p)$] Let $T$ be a reduced tree and $v$ be a vertex of $T$ of depth $p$. Then all descendants of $v$ have depth $<p$.
\item[$(\BB_p)$] Let $T_1,T_2$ be reduced trees. For $i=1,2$ let $v_i$ be an ordinary vertex of $T_i$. Assume that for $i=1,2$, $v_i$ has depth $p$, and all its descendants have depth $<p$. Further assume that $v_1$ is topologically equivalent to $v_2$.  Then $T_1(v_1)$ is isomorphic to $T_2(v_2)$.
\end{enumerate}
\end{prop}

\begin{proof}
The proof is by double induction on $p$. First observe that $(\AA_0)$ and $(\BB_0)$ are direct consequences of Lemma~\ref{lem:zero}. Our induction scheme is the following: prove $(\BB_1)$, then $(\AA_1)$, then $(\BB_2)$, then $(\AA_2)$ etc.

This reduces to two claims:

\paragraph{Claim 1} Assume $p\ge 1$. If $(\AA_k)$ and $(\BB_k)$ are true for all $k\le p$, then $(\BB_{p+1})$ is true.

Let $T_1,T_2,v_1,v_2$ be as in the statement of $(\BB_{p+1})$. From the induction hypothesis we know that $v_1,v_2$ have depth $p+1$ and their sons have depth at most $p$. Furthermore, the grandsons of $v_1,v_2$ have depth at most $p-1$.

Let $v'_1$ be a son of $v_1$ of depth $k$. For every point $x$ of $L_{v_1}$, there exists a sequence of points of $L_{v'_1}$ which converges to $x$. Since $v_1$ is equivalent to $v_2$, it follows that $v_2$ has a descendant $v'_2$ which is equivalent to $v'_1$. We claim that $v'_2$ is a son of $v_2$. Suppose not. Then there is a unique vertex $v''_2$ of $T_2$ which is a son of $v_2$ and an ancestor of $v'_2$, and the depth of $v'_2$ is strictly between the depths of $v_2$ and $v'_2$. Arguing as above, we see that $v''_2$ is equivalent to some descendant $v''_1$ of $v_1$. By the induction hypothesis, the subtrees $T_1(v'_1)$ and $T_2(v'_2)$ are isomorphic to each other, and so are $T_1(v''_1)$ and $T_2(v''_2)$. This contradicts Property~$(\RR_2)$.

To sum up, we have shown that for every son $v'_1$ of $v_1$, there is a son $v'_2$ of $v_2$ such that $v'_1$ is equivalent to $v'_2$. By the induction hypothesis, the subtrees $T_1(v'_1)$ and $T_2(v'_2)$ are isomorphic. By Property~$(\RR_3)$, for $i=1,2$ we have the following property: if $v',v''$ are two distinct sons of $v'_i$, then $T_i(v')$ is not isomorphic to $T_i(v'')$. Hence we have produced a bijection between the set of sons of $v_1$ and the set of sons of $v_2$ which preserves the isomorphism type of the generated trees. This implies that $T_1(v_1)$ is isomorphic to $T_2(v_2)$. We thus have proved Claim~1.

\paragraph{Claim 2} Assume $p\ge 1$. If $(\AA_k)$ is true for every $k<p$, and $(\BB_k)$ is true for every $k\le p$, then $(\AA_p)$ is true.

Arguing by contradiction as in the proof of Lemma~\ref{lem:zero} and taking a counterexample $(T,v)$ with $v$ of minimal combinatorial depth, we may assume the following: $v$ has depth $p$; it has at least one son $v_1$ of depth $p$, possibly other sons $v_2,\ldots,v_r$ of depth $p$, and possibly other sons of depth $<p$. Furthermore, all descendants of the $v_k$'s have depth $<p$.

There exists a sequence of points of $L_{v_1}$ converging to some point $x$ of $L_{v}$. This shows that $x$ is not isolated in $L(T)^{\ge p}$. Thus there exists a Cantor neighborhood $K\subset L(T)^{\ge p}$ all of whose points are equivalent in $L(T)$. This implies that $v_1$ is equivalent to $v$. By transitivity, all the $v_k$'s are equivalent to one another. By Property~$(\BB_p)$, their associated subtrees are pairwise isomorphic. This shows that $r=1$ (otherwise we would have a contradiction with Property~$(\RR_3)$.)

We now show that $v$ does not have a son of depth $<p$. Suppose that $v'$ is such a son. We are assuming $(\AA_k)$ for all $k<p$, so every descendant of $v'$ has depth strictly less than the depth of $v'$. Since $v$ is equivalent to $v_1$, there must exist a son $v_2$ of $v_1$ such that $v_2$ is equivalent to $v'$. Applying $(\BB_k)$ with $k$ the depth of $v'$, we deduce that $T(v')$ is isomorphic to $T(v_2)$. This contradicts Property~$(\RR_2)$.

Hence $v_1$ is the only son of $v$. By~$(\RR_1)$, it must be of type $\oo$. Hence every point of $L_{v_1}$ is isolated in $L(T)^{\ge p}$, but no point of $L_v$ is. This contradicts equivalence between $v$ and $v_1$. Thus we have proved Claim~2.
\end{proof}

We now turn to the proof of our main technical theorem. The proof is a variant of the argument used in the proof of Claim~1  above, with the difference that the root must be treated in a different way since Move~3 is not fully available for it.

\begin{proof}[Proof of Theorem~\ref{thm:reduced}]
Let $T_1,T_2$ be reduced admissible trees which are topologically equivalent. Since an admissible tree is reduced to the root iff its associated space is empty, we can assume that this is not the case.

Let $v_1$ be a son of the root of $T_1$. By topological equivalence, there exists an ordinary vertex $v_2$ of $T_2$ which is equivalent to $v_1$. By Proposition~\ref{prop:key}, the subtrees $T_1(v_1)$ and $T_2(v_2)$ are isomorphic.

We claim that $v_2$ is a son of the root of $T_2$. Suppose it is not. Let $v'_2$ be the only ancestor of $v_2$ which is a son of the root. By topological equivalence, there exists a vertex $v'_1$ of $T_1$ which is equivalent to $v'_2$. By Proposition~\ref{prop:key} again, $T_1(v'_1)$ is isomorphic to $T_2(v'_2)$. Hence $T_1(v'_1)$ has a proper subtree isomorphic to $T_1(v_1)$. This contradicts Property~$(\RR_2)$.

We have shown that the isomorphism types of the trees generated by the sons of the roots of $T_1,T_2$ are the same. By rule $(\RR_3)$, each type of subtree with  root of type $\Theta$ can occur at most once. This need not be true for types of subtrees with ancestor of type $\oo$. However, each son of the root of type $\oo$ contributes a single point, so the number of such points is invariant under homeomorphism. This concludes the proof of Theorem~\ref{thm:reduced}.
\end{proof}

\section{The general case}\label{sec:general}

In this section, we indicate how to adapt the proof from the planar case to the general case.

We begin with some simple observations. Let $\calx$ be a topological 2-automaton. Then every handle (resp.~cross-cap) in $M(\calx)$ must come from a handle (resp.~cross-cap) in some building block. Conversely, any handle (resp.~cross-cap) in some building block $X_k$ will give either infinitely many handles (resp.~cross-caps) in $M(\calx)$ or just one, according to whether there is a loop on $X_k$ or higher in the hierarchy, or there is no such loop.

It follows that we can determine the genus (finite or infinite) and orientability class of $M(\calx)$ directly on the automaton. In the sequel, we focus on the triple $(E,E',E'')$, using the machinery of decorated graphs. For brevity, we say that $(E,E',E'')$ is a \bydef{triple of spaces} if $E$ is a compact, metrizable, totally disconnected space, $E'$ is a closed subset of $E$, and $E''$ is a closed subset of $E'$. Two such objects $(E,E',E'')$ and $(F,F',F'')$ are \bydef{equivalent} if there is a homeomorphism $f:E\to E'$ such that $f(E')=F'$ and $f(E'')=F''$.

\subsection{Decorated graphs and admissible trees}
In order to suitably generalize our notion of decorated graph, we work with the extended set of symbols $S:=\{*,*^h,*^c,\oo,\oo^h,\oo^c,\Theta,\Theta^h,\Theta^c\}$. The `meaning' of these symbols will be explained next. For the moment, let us just say that the superscript $h$ stands for `handle', and `c' is for cross-cap.

A \bydef{decorated graph} is a pair $(G,f)$, where $G$ is a finite oriented graph without oriented cycles, endowed with a base vertex, called the \bydef{root}, and $f$ is a map from the set of all vertices of $G$ to the set $S$, in such a way that the root has image $*$, $*^h$, or $*^c$.

We would like to associate to a decorated graph $G$ a triple of spaces $(L(G),L'(G),L''(G))$. However, there is a difficulty, which we now explain by means of an example.

Let $G$ be a decorated tree with three vertices $v_1,v_2,v_3$ with $v_1$ the root, $v_2$ a son of $v_1$, and $v_3$ a son of $v_2$, and such that $v_1$ (resp.~$v_2$, resp.~$v_3$) has type $*$ (resp.~$\oo$, resp.~$\oo^h$.)

In order to define $L(G)$, we ignore the superscripts, so in this case we assign the letter $a$ to $v_2$ and the letter $b$ to $v_3$, and get a space with a single accumulation point $a^\infty$ and a sequence of isolated points of the form $a^k b^\infty$. Since only $v_2$ has a superscript $h$, it would be tempting to say that $L'(G)$ consists of all points of $L(G)$ except $a^\infty$. However, this does not work, since $L'(G)$ must be closed. The reader is invited to find a corresponding automaton (the connection between automata and decorated graphs will be explained in the next subsection), draw a picture of its associated surface, and check that every end of this surface is indeed nonplanar.

In order to overcome this issue, we make the following definition:

\newpage

\begin{defi}
A decorated graph $G$ is \bydef{coherent} if it satisfies the following properties:
\begin{enumerate}
\item No vertex of type $\oo$ or $\Theta$ has a descendant of type $*^h$, $*^c$, $\oo^h$, $\oo^c$, $\Theta^h$, or $\Theta^c$.
\item No vertex of type $\oo^h$ or $\Theta^h$ has a descendant of type $*^c$, $\oo^c$, or $\Theta^c$.
\end{enumerate}
\end{defi}

Let $G$ be a coherent decorated graph. We define a triple of spaces  $(L(G),L'(G),L''(G))$ as follows: set $r := n_\oo + 2 n_\Theta$, where $n_\oo$ (resp.~$n_\Theta$) is the number of vertices of $G$ of type belonging to the set $\{\oo,\oo^h,\oo^c\}$ (resp.~of type belonging to $\{\Theta,\Theta^h,\Theta^c\}$). Consider the alphabet $B=\{b_1,\ldots, b_r\}$ and assign to each vertex $v$ of $G$ zero, one or two letters according to its type. Then $L(G)$ is defined as a subspace of $B^\nn$ in the same way as in the planar case, ignoring the superscripts. 
The subset $L'(G)$ is defined as the union of the $L_v$'s over all vertices $v$ of type $\oo^h$, $\oo^c$, $\Theta^h$ or $\Theta^c$, and $L''(G)$ is the union of the $L_v$'s over all vertices of type $\oo^c$ or $\Theta^c$.

\begin{defi}
An \bydef{admissible tree} is a decorated graph which has the following properties:
\begin{enumerate}
\item It is a tree.
\item It is coherent.
\item All edges point away from the root.
\item The root has type $*$, and is the only vertex with this property.
\item There are no vertices of type $*^h$ or $*^c$.
\end{enumerate}
\end{defi}

\subsection{From automata to admissible trees}
Let $\calx$ be a topological 2-automaton. First we form the decorated graph $(G,f)$ from the directed graph underlying the automaton by assigning to each building block $X_k$ a symbol in $S$, in the following way: if there is no loop around $X_k$, then we put $*$ if $X_k$ is planar, $*^h$ if $X_k$ is orientable, but not planar, and $*^c$ if $X_k$ is nonorientable.
If there is one loop around $X_k$, then we put $\oo$ if $X_k$ is planar, $\oo^h$ if it is orientable, but not planar, and $\oo^c$ if it is nonorientable. If there are two or more loops, we put $\Theta$ if $X_k$ is planar, $\Theta^h$ if it is orientable but not planar, and $\Theta^c$ if it is nonorientable.

Then we modify $G$ in order to make it coherent. For brevity we say that a vertex \bydef{has type $\mathrm{T}^h$} if its type belongs to the set $\{*^h,\oo^h,\Theta^h\}$. Likewise,we say that a vertex has type $\mathrm{T}^c$ if its type belongs to the set $\{*^c,\oo^c,\Theta^c\}$. 

The procedure for making $G$ coherent has two steps.

\paragraph{Step 1} For every vertex $v$ of type $\oo$ (resp.~$\Theta$) that has at least one descendant of type $\mathrm{T}^h$, but no descendant of type $\mathrm{T}^c$, we change the type of $v$ to $\oo^c$ (resp.~$\Theta^c$).

\paragraph{Step 2} For every vertex $v$ of type $\oo$ (resp.~$\Theta$, resp.~$\oo^h$, resp.~$\Theta^h$) that has at least one descendant of type $\mathrm{T}^c$, we change the type of $v$ to $\oo^c$ (resp.~$\Theta^c$, resp.~$\oo^c$, resp.~$\Theta^c$).

The result is a coherent graph $G$ such that $(L(G),L'(G),L''(G))$ is equivalent to $(E(M(\calx)),E'(M(\calx)),E''(M(\calx))$.

Next we modify the decorated graph in order to make it a tree, by removing the part which is not accessible from the root, and duplicating some of its parts. This is done exactly as in the planar case, and preserves coherence.

Finally, in order to make the tree admissible, change all $*^h$ and $*^c$ to $*$ if necessary, and proceed as in the planar case.

To sum up, we have obtained an admissible tree $T$ such that  the triple $(L(T),L'(T),L''(T))$ is equivalent to $(E(M(\calx)),E'(M(\calx)),E''(M(\calx)))$.

Suppose now that we have two topological 2-automata $\calx_1,\calx_2$ and wish to know whether their associated surfaces are homeomorphic. We first check whether or not they have same genus and orientability class, as explained above. If they do, then we still have to check the equivalence of the triples $(L(T_i),L'(T_i),L''(T_i))$.

\subsection{Reduced trees and their classification}
We say that two admissible trees $T_1$ and $T_2$ are \bydef{equivalent} if the triple $(L(T_1),L'(T_1),L''(T_1))$ is equivalent to $(L(T_2),L'(T_2),L''(T_2))$. We also have the obvious extension of notion of (combinatorial) isomorphism between two admissible trees. Again, isomorphism implies equivalence; the converse does not hold, so we need to define reduced trees.

This is done using three moves which extend the moves used in the planar case. Moves 2 and 3 are defined in the same way as in the planar case, the word `isomorphic' being interpreted in the appropriate generalized way. Move~1 needs the following adjustment:

\paragraph{Move 1'} Let $v$ be a vertex of $T$ of type $\Theta$ (resp.~$\Theta^h$, resp.~$\Theta^c$). Let $v'$ be the father of $v$. Assume that $v'$ is not the root, and that $v$ is the only son of $v'$. Further assume that $v'$ has type $\oo$ or $\Theta$ (resp.~$\oo^h$ or $\Theta^h$, resp.~$\oo^c$ or $\Theta^c$.)
Remove the edge between $v$ and $v'$ and replace $v'$ by $v$.

Let $T$ be an admissible tree. If none of Moves 1', 2, 3 can be performed, we say that $T$ is \bydef{reduced}. Then we have the following extension of Proposition~\ref{prop:reducing}, which is proved in the same way:

\begin{prop}\label{prop:reducing general}
There is an algorithm which to any admissible tree $T$ associates a reduced admissible tree equivalent to $T'$.
\end{prop}

Theorem~\ref{thm:reduced} extends to the following result:
\begin{theo}\label{thm:reduced general}
Two reduced admissible trees $T_1,T_2$ are equivalent if and only if they are isomorphic.
\end{theo}

The proof of Theorem~\ref{thm:reduced general} is a straightforward extension of  that of Theorem~\ref{thm:reduced}. We need to define topological equivalence and depth of points.

Let $(X,X',X'')$ be a triple of spaces. We think of points of $X$ as colored by one of three colors according to the partition $X= X'' \sqcup (X'\setminus X'') \sqcup (X\setminus X')$. Thus, if $(X,X',X'')$ and $(Y,Y',Y'')$  are triples of spaces, a homeomorphism $f:X\to Y$ induces an equivalence of triples iff it is color-preserving. If $U$ is an open subset of $X$, then the restriction $f:U\to f(U)$ is still color-preserving. We say that a subset $A\subset X$ is \bydef{monochromatic} if all points of $A$ belong to the same of these three subsets. 

\begin{defi}
Let $(X,X',X'')$ and $(Y,Y',Y'')$ be triples of spaces. Let $x\in X$ and $y\in Y$. We say that $x$ is \bydef{equivalent} to $y$ if there exists a triple $(U,V,f)$ where $U$ is a neighborhood of $x$, $V$ a neighborhood of $y$, and $f:U\to V$ is a color-preserving homeomorphism which takes $x$ to $y$.
\end{defi}

A point $x\in X$ has \bydef{depth} 0 if either $x$ is isolated, or $x$ has an open, monochromatic neighborhood homeomorphic to the Cantor set. Assuming we know what `to have depth $k$' means for all $k$ less than some number $p$, we say that $x$ has \bydef{depth} $p$ if it belongs to the complement $A$ of the set of points of depth $<p$, and either it is isolated in $A$, or it has an open neighborhood $K$ in $A$ which is a monochromatic Cantor set, and all of whose points are equivalent to $x$ in $X$.

In order to motivate the previous definition, consider the following example:

\begin{example}
Let $T_7$ be an admissible tree with three vertices $*_0,v,w$ where $v$ is the only son of $*_0$ and has type $\Theta^c$, and $w$ is the only son of $v$ and has type $\Theta$. In order to describe $L(T_7)$, we introduce two letters $a,b$ corresponding to $v$ and two letters $c,d$ corresponding to $w$. Then $L(T_7)=U \cup V$ where $U$ is the set of words containing only $a$'s and $b$'s, and $V$ its complement.

Since $V$ is a countable union of monochromatic Cantor sets, all of its points have depth 0. However, any point $x$ of $U$ has depth 1. Indeed, although $x$ certainly has open neighborhoods homeomorphic to the Cantor set (e.g.~$L(T_7)$ itself), it does not have any monochromatic one, since $x$ belongs to $L'(T_7)$, and there are sequences of points not in $L'(T_7)$ converging to $x$.
\end{example}

As in the planar case, one proves that two equivalent points have the same depth (cf.~Lemma~\ref{lem:local}), that it makes sense to talk about the topological type of a vertex of an admissible tree (cf.~Lemma~\ref{lem:ind}), and that when the tree is reduced, the points of depth ~0 are exactly the leaves (cf.~Lemma~\ref{lem:zero}.)
The key proposition is the following extension of Proposition~\ref{prop:key}:

\begin{prop}\label{prop:keygen}
For every integer $p\ge 0$, the following assertions hold:
\begin{enumerate}
\item[$(\AA_p)$] Let $T$ be a reduced tree and $v$ be a vertex of $T$ of depth $p$. Then all descendants of $v$ have depth $<p$.
\item[$(\BB_p)$] Let $T_1,T_2$ be reduced trees. For $i=1,2$ let $v_i$ be an ordinary vertex of $T_i$. Assume that for $i=1,2$, $v_i$ has depth $p$, and all its descendants have depth $<p$. Further assume that $v_1$ is topologically equivalent to $v_2$.  Then $T_1(v_1)$ is isomorphic to $T_2(v_2)$.
\end{enumerate}
\end{prop}

\begin{proof}
The only significant difference with the planar case being in the replacement of Move~1 by Move~1', we discuss the part of the proof concerned with that Move in detail, and give only a brief sketch for the rest.

As in the planar case, $(\AA_0)$ and $(\BB_0)$ are immediate from the characterization of points depth~0, and the induction is based on two claims:

\paragraph{Claim 1} Assume $p\ge 1$. If $(\AA_k)$ and $(\BB_k)$ are true for all $k\le p$, then $(\BB_{p+1})$ is true.

Let $T_1,T_2,v_1,v_2$ be as in the statement of $(\BB_{p+1})$. From the induction hypothesis we know that $v_1,v_2$ have depth $p+1$ and their sons have depth at most $p$. Furthermore, the grandsons of $v_1,v_2$ have depth at most $p-1$.

Let $v'_1$ be a son of $v_1$ of depth $k$. From equivalence of $v_1$ and $v_2$ we see that it $v_2$ has a descendant $v'_2$ which is equivalent to $v'_1$. Using  Property~$(\RR_2)$ and the induction hypothesis, we show that $v'_2$ is actually a son of $v_2$. Using Property~$(\RR_3)$, we get a bijection from the set of sons of $v_1$ and the set of sons of $v_2$ which preserves the isomorphism type of the generated trees. This proves Claim~1.

\paragraph{Claim 2} Assume $p\ge 1$. If $(\AA_k)$ is true for every $k<p$, and $(\BB_k)$ is true for every $k\le p$, then $(\AA_p)$ is true.

Arguing by contradiction as in the proof of Lemma~\ref{lem:zero} and taking a counterexample $(T,v)$ with $v$ of minimal combinatorial depth, we may assume the following: $v$ has depth $p$; it has at least one son $v_1$ of depth $p$, possibly other sons of depth $p$, and possibly other sons of depth $<p$. Furthermore, all descendants of the $v_k$'s have depth $<p$.

Let $x$ be a point of $L_{v}$. Then $x$ is not isolated in $L(T)^{\ge p}$, so there exists a monochromatic Cantor neighborhood $K\subset L(T)^{\ge p}$ all of whose points are equivalent in $L(T)$. Moreover, $v_1$ is equivalent to $v$. Using Properties~$(\RR_3)$ and~$(\BB_p)$ we see that $v$ has only one son of depth $p$.

The possibility that $v$ should have a son of depth $<p$ is ruled out as in the planar case. Hence $v_1$ is the only son of $v$. We now discuss according to the type of $v_1$ and show that each case leads to a contradiction.

\paragraph{Case 1} $v_1$ has type $\oo$, $\oo^h$, or $\oo^c$.

Then $v$ has depth $p+1$.

\paragraph{Case 2} $v_1$ has type $\Theta^c$.

Then by coherence, $v_1$ has type $\oo^c$ or $\Theta^c$, so the tree is not reduced (using Move~1'.)

\paragraph{Case 3} $v_1$ has type $\Theta^h$.

Then by coherence, $v$ has type either in $\{\oo^h,\Theta^h\}$, in which case the tree is not reduced, or in $\{\oo^c,\Theta^c\}$, in which case $v$ has depth $p+1$.

\paragraph{Case 4} $v_1$ has type $\Theta$.

Then $v$ has type either in $\{\oo^h,\Theta^h,\oo^c,\Theta^c\}$, in which case the tree is not reduced, or in $\{\oo,\Theta\}$, in which case $v$ has depth $p+1$.
\end{proof}

Finally, one deduces Theorem~\ref{thm:reduced general} from Proposition~\ref{prop:keygen} in the same way that Theorem~\ref{thm:reduced} was deduced from Proposition~\ref{prop:key}.

\section{Concluding remarks}

\begin{rem}
Our definition of a topological automaton is rather restrictive. It is possible to broaden it, for instance by removing the restriction that $k\le l$ in the last condition. This creates technical problems, but does not seem to enlarge the class of manifolds significantly.

At one extreme, one may think of associating an $n$-dimensional manifold to a Turing machine. One simple-minded way to do this is to start with a $n$-disk; each time the Turing machine does something, add an $n$-annulus; if the machine stops, glue in a $n$-disk. Then the resulting $n$-manifold is compact if and only if the Turing machine stops. Since the halting problem for Turing machines is undecidable, the homeomorphism problem for $n$-manifolds arising from this construction is undecidable.

When $n=1$, given the fact that there are only two $1$-manifolds up to homeomorphism, and that they are both automata $1$-manifolds, it is tempting to dismiss this phenomenon as an artefact caused by the use of Turing machines. In higher dimensions, however, this should probably be taken seriously. Notions of \emph{computable manifolds} have been proposed; we refer to work of Calvert-Miller~\cite{cm:computable} and Aguilar-Conde~\cite{ac:computable} and the references therein for further discussion.

It would be interesting to define intermediate classes between automata surfaces and computable surfaces and determine whether the homeomorphism problem---or other algorthmic problems---are decidable for them.
\end{rem}

\begin{rem}
Defining the \bydef{complexity} of an automaton surface $F$ as the number of ordinary vertices of the unique reduced admissible tree corresponding to $F$, it is possible to enumerate automata surfaces in order of increasing complexity in the spirit of S.~Matveev's enumeration of 3-manifolds~\cite{matveev:complexity}. We hope that this could be useful to test conjectures or search for a counterexample in a systematic way.

Unsurprisingly, the examples mentioned in the introduction have low complexity. The first infinite-type planar surface appearing in this enumeration that currently does not have a name has a Cantor set of ends, all of whose points are limits of sequences of isolated ends. Perhaps it should be called the Cantor flute surface?
\end{rem}

\begin{rem}
For $n\ge 4$, the homeomorphism problem for automata $n$-manifolds is of course undecidable, since this is already the case for compact $n$-manifolds by Markov's theorem. For $n=3$ it is an open question; in fact it is already open in the case of 1-ended 3-manifolds (presented by topological 3-automata with only two states and two arrows). We refer the interested reader to the paper~\cite{maillot:one} for further discussion.
\end{rem}

\appendix
\section{Uncountably many surfaces}\label{sec:appendix}

In this appendix we prove the following proposition.\footnote{Since writing the first version of this article, we have found the reference~\cite{reichbach:power}. We have decided to keep this appendix since it is referred to in the text.}
\begin{prop}\label{prop:uncountable}
There are uncountably many planar surfaces up to homeomorphism.
\end{prop}

\begin{proof}
We use a construction which was shown to us by Gilbert Levitt. Let $C$ be a Cantor set embedded in $[0,1]\times \{0\} \subset \rr^2$. Fix an arbitrary injective convergent sequence $((x_n,0))_{n\in\nn}$ of points of $C$ and let $(x_\infty,0)\in C$ be its limit. For each natural number $n$, fix an increasing embedding $\iota_n:\omega^n+1\to [0,1]$ such that $1$ belongs to the image of $\iota_n$; define an embedding $\sigma_n$ of $[0,1]$ into $\rr^2$ by setting $\sigma_n(t)=(x_n,(1-t)/(n+1))$; finally let $X_n$ be the image of $\sigma_n \circ \iota_n$.

We now define a map $\phi:\{0,1\}^\nn \to \mathcal{P}(\rr^2)$ by the formula $$\phi( (a_n)_{n\in\nn}) := C \cup \bigcup_{n\in\nn, a_n=1} X_n.$$ 

Note that for each $\mathbf{a}=(a_n)_{n\in\nn}$, the space $\phi(\mathbf{a})$ is compact and totally disconnected. (For compactness, observe that $C$ and every finite union of $X_n$ are compact, while if $(P_k)$ is a sequence of points of $\phi(\mathbf{a})$ that hits $X_n$ for infinitely many $n$, then it admits a subsequence that converges to $(x_\infty,0)$, since the diameter of $X_n$ tends to $0$ as $n$ tends to infinity.)

\begin{lem}
Let $\textbf{a},\textbf{b}\in \{0,1\}^\nn$. If $\phi(\textbf{a})$ and $\phi(\textbf{b})$ are homeomorphic, then $\textbf{a}=\textbf{b}$.
\end{lem}

\begin{proof}
Let $f$ be a homeomorphism from $\phi(\textbf{a})$ to $\phi(\textbf{b})$.
Recall that a point in a totally disconnected space is called a \bydef{condensation point} if all of its neighborhoods are uncountable. Observe that the set of condensation points of $\phi(\textbf{a})$ (resp.~$\phi(\textbf{b})$) is exactly $C$. Hence $f(C)=C$. The other points of $\phi(\textbf{a})$ have finite Cantor-Bendixson rank, and the rank is preserved by $f$, i.e.~$f$ sends isolated points to isolated points, limits of sequences of isolated points to similar points etc. Thus if $a_n=1$ for some $n$, then $f((x_n,0))=(x_n,0)$ and $b_n=1$. Conversely, if $b_n=1$ then $a_n=1$.
\end{proof}

Proposition~\ref{prop:uncountable} now follows since for every $\textbf{a}\in \{0,1\}^\nn$ the surface $\rr^2\setminus \phi(\textbf{a})$ has $\phi(\textbf{a}) \sqcup \{*\}$ as space of ends.

\end{proof}

\bibliographystyle{abbrv}
\bibliography{hom}

Institut Montpelli\'erain Alexander Grothendieck,
CNRS - Universit\'e de Montpellier.\\ 
\texttt{sylvain.maillot@umontpellier.fr}

\end{document}